\documentclass[11pt]{article}

\usepackage[T1]{fontenc}
\usepackage[utf8]{inputenc}
\usepackage{lmodern}
\usepackage{amsmath,amssymb,amsthm,mathtools}
\usepackage{microtype}
\usepackage[margin=1.05in]{geometry}
\usepackage{enumitem}
\usepackage{needspace}
\usepackage{tikz}
\usepackage{xcolor}
\usepackage[colorlinks=true,linkcolor=blue!50!black,citecolor=blue!50!black,urlcolor=blue!50!black]{hyperref}
\hypersetup{
  pdftitle={Functional inequalities along Wasserstein geodesics},
  pdfauthor={Nathael Gozlan, Hugo Malamut, Irène Waldspurger},
  pdfkeywords={optimal transport, Wasserstein geodesics, transport-entropy inequalities, Poincaré inequalities, Cheeger inequality, strong log-concavity, displacement convexity}
}

\newtheorem{theorem}{Theorem}[section]
\newtheorem{proposition}[theorem]{Proposition}
\newtheorem{lemma}[theorem]{Lemma}
\newtheorem{corollary}[theorem]{Corollary}
\theoremstyle{remark}
\newtheorem{remark}[theorem]{Remark}

\newcommand{\R}{\mathbb R}
\newcommand{\Pp}{\mathcal P}
\newcommand{\id}{\mathrm{Id}}
\newcommand{\Ent}{\mathrm H}
\newcommand{\dd}{\,\mathrm d}

\title{Functional inequalities along Wasserstein geodesics}
\author{%
Nathael Gozlan\thanks{Université Paris Cité, CNRS, MAP5, F-75006 Paris, France.
Email: \href{mailto:nathael.gozlan@u-pariscite.fr}{\nolinkurl{nathael.gozlan@u-pariscite.fr}}.},
Hugo Malamut\thanks{Université Paris Cité, CNRS, MAP5, F-75006 Paris, France.
Email: \href{mailto:hugo.malamut@u-pariscite.fr}{\nolinkurl{hugo.malamut@u-pariscite.fr}}.},
Irène Waldspurger\thanks{DMA, École normale supérieure, Université PSL, CNRS,
45 rue d'Ulm, 75230 Paris Cedex 05, France.
Email: \href{mailto:Irene.Waldspurger@ens.fr}{\nolinkurl{Irene.Waldspurger@ens.fr}}.}}
\date{}

\begin{document}
\maketitle

\begin{abstract}
We study functional inequalities along Wasserstein geodesics. If $\mu_0$ and $\mu_1$ are respectively $\kappa_0$- and $\kappa_1$-strongly log-concave probability measures on $\mathbb R^n$, we prove that their quadratic Wasserstein geodesic satisfies
\[
 W_2(\mu_t,\nu)\leq\sqrt{2}\left(\frac{1-t}{\sqrt{\kappa_0}}+
 \frac{t}{\sqrt{\kappa_1}}\right)\sqrt{\mathrm H(\nu\mid\mu_t)},
\]
for all probability measures $\nu$ on $\R^n.$
The coefficient is sharp. By linearization, this recovers the Poincar\'e estimate of Han and Zhu \cite{HanZhu2026}. On the real line, we prove convexity of the square roots of the optimal $T_1$ and $T_2$ constants along monotone interpolation between arbitrary probability measures. The argument applies to more general transport entropy inequalities. We also establish convexity of the rescaled $L^p$ Poincar\'e constants for every finite $p\geq1$, including the square root of the Poincar\'e constant and the inverse Cheeger constant. 
Finally, we construct a planar Wasserstein geodesic whose endpoints satisfy \(T_2\) and all finite-\(p\) \(L^p\)-Poincaré inequalities, whereas every interior interpolant fails these inequalities.
\end{abstract}

\begingroup
\par
\endgroup

\section{Introduction}

The aim of this paper is to study how the optimal constants of functional inequalities behave along Wasserstein geodesics. We prove two kinds of results. On the real line, Poincar\'e and Talagrand constants are shown to be convex along monotone interpolation. In arbitrary dimension, strong log-concavity of the endpoints gives a sharp Talagrand inequality along the geodesic, with a bound determined by the endpoint curvature lower bounds. In dimension two and higher, we also give a counterexample to the convexity of the Talagrand and Poincar\'e constants along Wasserstein geodesics. We begin with a presentation of these results and their motivations. The main objects (spaces of probability measures, Wasserstein distances, relative entropy, Poincar\'e constant) are recalled in the Notation section below.

The main motivation comes from the work of Aishwarya and Rotem \cite{AishwaryaRotem2026}, who use convexity properties of relative entropy and of Poincaré constant to prove dimensional Brunn--Minkowski inequalities for the Gaussian measure (\cite{EM21},\cite{KL21}). More precisely, Question~1.6 of \cite{AishwaryaRotem2026} asks whether Wasserstein interpolants between even $1$-strongly log-concave measures satisfy the Poincar\'e inequality with constant $1$ for odd functions. They prove this in dimension one and in special Gaussian cases. For the Brunn--Minkowski application however, the coupling need not be optimal: Aishwarya and Li \cite{AishwaryaLi2025} use generalized geodesics with a common Gaussian source to obtain an interpolation satisfying the full Poincar\'e inequality.

Han and Zhu \cite{HanZhu2026} recently answered Aishwarya and Rotem's question in arbitrary dimension. They even provide the following sharp estimate: suppose that $\mu_i=e^{-V_i}\dd x$, $i=0,1$, are probability measures on $\R^n$, where $V_i$ is proper and lower semicontinuous and $V_i-\kappa_i|\cdot|^2/2$ is convex for some $\kappa_i>0$, then
\begin{equation}\label{eq:Han-Zhu}
 \sqrt{C_P(\mu_t)}\leq
 \frac{1-t}{\sqrt{\kappa_0}}+\frac{t}{\sqrt{\kappa_1}},
\end{equation}
with $C_P$ the Poincar\'e constant and $(\mu_t)$ the Wasserstein geodesic between $\mu_0$ and $\mu_1$. Estimate \eqref{eq:Han-Zhu} is an upper bound in terms of the endpoint curvature lower bounds. It does not assert convexity of the optimal Poincar\'e constant along the geodesic, which is false in general, as shown in \cite{HanZhu2026}.

The present paper is a substantially revised and extended version of the last chapter of \cite{Malamut2025}, where the second author raised the broader question of interpolating the constants of several functional inequalities along Wasserstein geodesics. The main result of the paper, Theorem~\ref{thm:curvature-T2}, provides, under the same assumptions as \eqref{eq:Han-Zhu}, a sharp Talagrand inequality for the interpolated measure:
\begin{equation}\label{eq:intro-curvature}
 W_2(\mu_t,\nu)\leq\sqrt{2}\left(
 \frac{1-t}{\sqrt{\kappa_0}}+\frac{t}{\sqrt{\kappa_1}}\right)
 \sqrt{\Ent(\nu\mid\mu_t)},
\end{equation}
for every probability measure $\nu$ and $t\in (0,1)$. Equality is attained for Gaussian endpoints with covariance matrices $\kappa_i^{-1}I$. 
It is well known that the linearization of a Talagrand inequality with constant $C$ yields a Poincar\'e inequality with constant $C/2$, see e.g. \cite{GozlanLeonard2010}. Thus \eqref{eq:intro-curvature} is strictly stronger than \eqref{eq:Han-Zhu}, which it recovers by linearization.
\medskip

Beyond those considerations about log-concave endpoints, we also provide general convexity results applying to arbitrary probability measures in dimension $1$. On the real line, we obtain estimates in terms of the optimal endpoint constants themselves. For $1\leq p\leq2$, let $\mathsf T_p(\mu)$ denote the best constant in
\[
 W_p(\mu,\nu)^2\leq\mathsf T_p(\mu)\Ent(\nu\mid\mu),\qquad \forall \nu.
\]
In particular, $\mathsf T_1$ and $\mathsf T_2$ are the optimal constants in the usual $T_1$ and $T_2$ inequalities. We prove in Corollary~\ref{cor:Talagrand} that
\begin{equation}\label{eq:intro-Tp-scale}
 \sqrt{\mathsf T_p(\mu_t)}
 \leq(1-t)\sqrt{\mathsf T_p(\mu_0)}+t\sqrt{\mathsf T_p(\mu_1)},\qquad t\in(0,1),
\end{equation}
where $(\mu_t)$ is the monotone interpolation between the probability measures $\mu_0$ and $\mu_1$ on $\R$.

As discovered by Marton \cite{Marton1986,Marton1996} and Talagrand \cite{Talagrand1996}, these transport-entropy inequalities are closely connected to the concentration of measure phenomenon \cite{Ledoux}. The inequality $T_1$ was shown to be equivalent to exponential integrability of order $2$ \cite{DGW04}. The inequality $T_2$ is equivalent to dimension-free Gaussian concentration \cite{Gozlan2009}. A dual formulation of transport-entropy inequalities was developed by Bobkov and G\"otze \cite{BobkovGotze1999}, and their relation to logarithmic Sobolev type inequalities was first investigated by Otto and Villani \cite{OttoVillani2000} for the $T_2$ inequality, see also \cite{BobkovGentilLedoux2001}. On the line, a characterization for the $T_2$ inequality and more general transport-entropy inequalities associated with convex costs that are quadratic near zero is given in \cite{Gozlan2012}. We refer to \cite{GozlanLeonard2010} for a survey on these inequalities.

Our one-dimensional argument applies to a larger class of Talagrand type inequalities. Given an even convex cost $\alpha$ with $\alpha(0)=0$, write
\[
 \mathcal T_\alpha(\mu,\nu)
 =\inf_{\pi\in\Pi(\mu,\nu)}\int\alpha(x-y)\dd\pi(x,y).
\]
For a non-decreasing function $\psi:[0,\infty)\to[0,\infty)$ with $\psi(0)=0$ and $\psi(s)>0$ for $s>0$, define
\begin{equation}\label{eq:intro-Kalpha}
 \mathsf T_{\alpha,\psi}(\mu)
 =\sup_{0<\Ent(\nu\mid\mu)<\infty}
 \frac{\mathcal T_\alpha(\mu,\nu)}{\psi(\Ent(\nu\mid\mu))}.
\end{equation}
Theorem~\ref{thm:Kalpha-convex} states that $\mathsf T_{\alpha,\psi}$ is convex along monotone interpolation. For homogeneous costs, a scaling argument gives convexity of the corresponding root. Choosing $\alpha(z)=|z|^p$ and $\psi(s)=s^{p/2}$ gives back \eqref{eq:intro-Tp-scale}, see Remark \ref{rem:homogeneity}.

\medskip

We also consider the optimal constants of the $L^p$ Poincar\'e inequalities. For $\mu\in\Pp(\R^n)$ and $1\leq p<\infty$, let $\mathsf C_p(\mu)$ be the best constant in
\begin{equation}\label{eq:intro-Cp}
 \inf_{a\in\R}\int|f-a|^p\dd\mu
 \leq\mathsf C_p(\mu)\int|\nabla f|^p\dd\mu,
\end{equation}
for $f\in C^1(\R^n)$ with $\nabla f\in C_c(\R^n;\R^n)$.\footnote{\label{fn:test-class}Spatial cutoff and mollification give the same constant for $C_c^\infty(\R^n)$ test functions and extend \eqref{eq:intro-Cp} to bounded $C^1$ functions with bounded gradient, for any probability measure $\mu$.}
For any two probability measures $\mu_0$ and $\mu_1$ on $\R$, Theorem~\ref{thm:sobolev-arbitrary} gives
\begin{equation}\label{eq:intro-Cp-scale}
 \mathsf C_p(\mu_t)^{1/p}
 \leq(1-t)\mathsf C_p(\mu_0)^{1/p}
       +t\mathsf C_p(\mu_1)^{1/p},\qquad t\in(0,1).
\end{equation}
For $p=2$, this is convexity of the square root of the Poincar\'e constant $C_P$, while for $p=1$, it corresponds to the inverse Cheeger constant. This extends the one-dimensional result of Aishwarya and Rotem \cite[Theorem~4.11]{AishwaryaRotem2026} who proved \eqref{eq:intro-Cp-scale} for $p = 2$, symmetric measures and odd test functions.

Finally, we construct a planar example in which the endpoint measures have connected supports and obey Talagrand $T_2$ and $L^p$ Poincar\'e inequalities and we show that the support of the interpolating measure has two separated connected components and, thus, does not satisfy those functional inequalities. This is a new geometric counterexample in dimension two, in the spirit of Santambrogio and Wang \cite{SantambrogioWang2016}, who show that displacement interpolation need not preserve convexity of supports. That it may also disconnect the support was already familiar to practitioners: in computational anatomy, optimal transport is known not to preserve the topology of the interpolated shapes \cite{Feydy2020}.

The paper is organized as follows. Section~\ref{sec:curvature-T2} proves the sharp Talagrand inequality \eqref{eq:intro-curvature} in terms of the endpoint curvature lower bounds. Sections~\ref{sec:transportentropy} and~\ref{sec:sobolev} treat Talagrand and Poincaré constants on the line. Section~\ref{sec:counter-example} gives the planar counterexample.

\subsection*{Notation}
Throughout the paper, $\Pp(\R^n)$ denotes the set of Borel probability measures on $\R^n$ and, for $p\geq1$, $\Pp_p(\R^n)$ denotes the subset of those having a finite moment of order $p$. The set of couplings of $\mu,\nu\in\Pp(\R^n)$, that is of probability measures on $\R^n\times\R^n$ having $\mu$ and $\nu$ as marginals, is denoted by $\Pi(\mu,\nu)$, and the Wasserstein distance of order $p$ is defined on $\Pp_p(\R^n)$ by
\[
 W_p^p(\mu,\nu)=\inf_{\pi\in\Pi(\mu,\nu)}\int|x-y|^p\dd\pi(x,y),
\]
where $|\cdot|$ is the Euclidean norm. We use the same formula for arbitrary probability measures, allowing the value $+\infty$. If $T$ is a measurable map and $\mu$ a probability measure, $T_\#\mu$ denotes the image of $\mu$ under $T$. The relative entropy of $\nu$ with respect to $\mu$ is defined by
\[
 \Ent(\nu\mid\mu)=\int\log\left(\frac{\mathrm d\nu}{\mathrm d\mu}\right)\dd\nu
\]
if $\nu$ is absolutely continuous with respect to $\mu$, and $\Ent(\nu\mid\mu)=+\infty$ otherwise. All the optimal constants considered below take their values in $[0,+\infty]$.

For $p=2$, the infimum in \eqref{eq:intro-Cp} is attained at $a=\int f\dd\mu$ and equals $\int|f-\int f\dd\mu|^2\dd\mu$. Thus $\mathsf C_2(\mu)$ is $C_P(\mu)$, the Poincar\'e constant of $\mu$. For $p=1$, $\mathsf C_1(\mu)$ is the inverse of the  Cheeger isoperimetric constant $h(\mu)$, defined by
\[
h(\mu)=\inf_A\frac{\mu^+(A)}{\min\{\mu(A),1-\mu(A)\}}.
\]
The infimum runs over Borel sets $A$ with $0<\mu(A)<1$, and
\[
 \mu^+(A)=\liminf_{r\to0^+}\frac{\mu(A^r)-\mu(A)}{r},
 \qquad A^r=\{x:\operatorname{dist}(x,A)<r\}.
\]
The identity $\mathsf C_1=h^{-1}$ is the classical equivalence with the median $L^1$ Poincaré inequality, see \cite[Lemma~3.1 and Remark~3.5]{BobkovHoudre1997}. 

\section{Talagrand inequality under curvature lower bounds}
\label{sec:curvature-T2}

In this section, we prove the Talagrand estimate \eqref{eq:intro-curvature}. Throughout the section, we will say that a probability measure $\mu$ on $\R^n$ is $\kappa$-strongly log-concave, with $\kappa>0$, if it is of the form $\dd\mu=e^{-V}\dd x$ with $V:\R^n\longrightarrow(-\infty,+\infty]$, proper and lower semicontinuous, and such that $V-\frac{\kappa}{2}|\cdot|^2$ is convex.

Recall that, for $\mu\in\Pp(\R^n)$, we denote by $\mathsf T_2(\mu)$ the best constant in the inequality $W_2^2(\mu,\nu)\leq\mathsf T_2(\mu)\Ent(\nu\mid\mu)$.

\begin{theorem}[Talagrand interpolation under curvature lower bounds]
\label{thm:curvature-T2}
Assume that $\mu_i$ is $\kappa_i$-strongly log-concave  with $\kappa_i>0$, $i=0,1$. Denote by $(\mu_t)$ the Wasserstein geodesic between them. Then, for all $t\in[0,1]$ and all probability measures $\nu$ on $\R^n$,
\begin{equation}\label{eq:curvature-T2}
  W_2(\mu_t,\nu)
  \leq\sqrt{2}\left(
    \frac{1-t}{\sqrt{\kappa_0}}+
    \frac{t}{\sqrt{\kappa_1}}
  \right)\sqrt{\Ent(\nu\mid\mu_t)}.
\end{equation}
Equivalently,
\begin{equation}\label{eq:curvature-T2-constant}
  \mathsf T_2(\mu_t)
  \leq2\left(
    \frac{1-t}{\sqrt{\kappa_0}}+
    \frac{t}{\sqrt{\kappa_1}}
  \right)^2.
\end{equation}
Moreover, the constant in \eqref{eq:curvature-T2-constant} is optimal.
\end{theorem}

The proof of Theorem~\ref{thm:curvature-T2} relies on the following  strengthened form of Talagrand inequality due to Kolesnikov \cite[Section~2]{Kolesnikov2004}. It compares two gradients of convex functions defined on a common source measure. Kolesnikov also used strengthened transport inequalities of this type to derive Sobolev regularity estimates for optimal transport maps onto uniformly log-concave measures \cite[Section~4]{Kolesnikov2013}. For completeness, we recall a proof based on convexity of relative entropy along generalized geodesics.

\begin{lemma}[Two-map Talagrand inequality]
\label{lem:Kolesnikov-two-map}
Let $\eta\in\Pp_2(\R^n)$ be absolutely continuous with respect to the Lebesgue measure, and let $\nabla a$ and $\nabla b$ be gradients of convex functions such that $(\nabla a)_\#\eta=\mu$ and $(\nabla b)_\#\eta=\nu$, with $\mu,\nu\in\Pp_2(\R^n)$. If $\mu$ is $\kappa$-strongly log-concave for some $\kappa>0$, then
\begin{equation}\label{eq:Kolesnikov-two-map}
  \int|\nabla a-\nabla b|^2\dd\eta
  \leq \frac{2}{\kappa}\Ent(\nu\mid\mu).
\end{equation}
\end{lemma}

\begin{proof}
We may assume that $\Ent(\nu\mid\mu)<\infty$. For $s\in[0,1]$, define $\nu_s=((1-s)\nabla a+s\nabla b)_\#\eta$, the so-called generalized geodesic from $\nu_0=\mu$ to $\nu_1 = \nu$ with base $\eta.$
We use the following decomposition of relative entropy: 
\[
\Ent(p \mid \mu) =  \Ent(p \mid \mathrm{Leb})+\int V\dd p,\qquad \forall p \in \mathcal{P}_2(\R^n),
\]
where $\Ent(p \mid \mathrm{Leb})= \int \log \frac{\dd p}{\dd x} \dd p$ if $p$ is absolutely continuous with respect to Lebesgue and $+\infty$ otherwise and where $V:\R^n \to \R \cup \{+\infty\}$ is the $\kappa$-convex function such that $\dd \mu = e^{-V}\,\dd x$. 
According to \cite[Proposition~9.3.2]{AGS2005}, the ``potential energy'' $p \mapsto \int V\dd p$ is $\kappa$-convex and, according to \cite[Proposition~9.3.9]{AGS2005} (applied with $F(x)= x\log x$, $x>0$) the ``internal energy'' $p\mapsto \Ent(p \mid \mathrm{Leb})$ is convex along the generalized geodesic $(\nu_s)_{s\in[0,1]}$. 
Since $\nu_0=\mu$, this gives
\[
  0\leq\Ent(\nu_s\mid\mu)
  \leq s\Ent(\nu\mid\mu)
       -\frac{\kappa}{2}s(1-s)\int|\nabla a-\nabla b|^2\dd\eta,
  \qquad s\in[0,1].
\]
Dividing by $s>0$ and letting $s\to0$ gives \eqref{eq:Kolesnikov-two-map}.
\end{proof}

Choosing $\eta=\mu$, $\nabla a=\id$ and $\nabla b$ the Brenier map from $\mu$ to $\nu$ in Lemma~\ref{lem:Kolesnikov-two-map}, we recover the classical inequality
\begin{equation}\label{eq:classical-curvature-T2}
  W_2^2(\mu,\nu)\leq\frac{2}{\kappa}\Ent(\nu\mid\mu).
\end{equation}
We refer to \cite{CE02} for a simple transport proof of \eqref{eq:classical-curvature-T2} and to \cite{CMS01, LV09, Stu06a} for other applications of geodesic convexity in terms of functional inequalities.

\begin{proof}[Proof of Theorem~\ref{thm:curvature-T2}]
Let $\nabla\Phi$ be the Brenier map from $\mu_0$ to $\mu_1$, and let $ F_t:=(1-t)\id+t\nabla\Phi$, so that 
\begin{equation}\label{eq:Ft-curvature-T2}
  \mu_t=(F_t)_\#\mu_0,\qquad 0\leq t\leq1.
\end{equation}
Fix $0<t<1$ and a probability measure $\nu$ such that $H:=\Ent(\nu\mid\mu_t)<\infty$. The measures $\mu_0$ and $\mu_1$ have a second order exponential moment, and so does $\mu_t$, by convexity of $|\cdot|^2$ and H\"older's inequality. This easily implies that $\nu\in\Pp_2(\R^n)$.

The monotonicity of $\nabla\Phi$ implies that $F_t^{-1}$ is $(1-t)^{-1}$-Lipschitz on its domain. Since both endpoints are absolutely continuous, $\nabla\Phi$ also has an inverse $\nabla\Phi^*$ defined $\mu_1$ almost everywhere. Define
\[
 \nu_0=(F_t^{-1})_\#\nu,
 \qquad \nu_1=(\nabla \Phi)_\#\nu_0.
\]
Invariance of entropy under invertible maps gives
\begin{equation}\label{eq:three-entropies}
 \Ent(\nu_0\mid\mu_0)=\Ent(\nu_1\mid\mu_1)=H.
\end{equation}
In particular, $\nu_0$ and $\nu_1$ are absolutely continuous and have finite second moments.

Let $\nabla b$ be the Brenier map from $\mu_0$ to $\nu_1$.

Take $X_0\sim\mu_0$ and set $X_1:=\nabla\Phi(X_0)$, $Y_1:=\nabla b(X_0)$, and $Y_0:=\nabla\Phi^*(Y_1)$.
Then $X_i\sim\mu_i$ and $Y_i\sim\nu_i$ for $i=0,1$. The couplings used in the proof are represented by
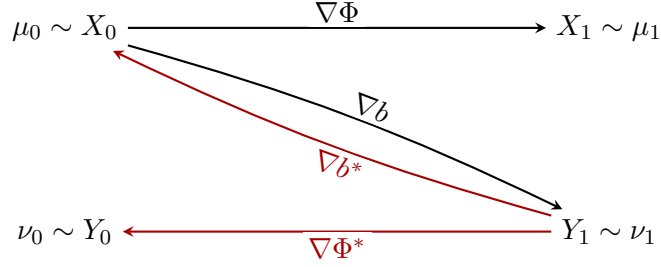
\begin{center}
\begin{tikzpicture}[
  >=stealth,
  baseline=(current bounding box.center),
  forward/.style={->,thick},
  backward/.style={->,thick,red!65!black},
  maplabel/.style={fill=white,inner sep=2pt}
]
  \node (x0) at (0,1.35) {$\mu_0\sim X_0$};
  \node (x1) at (7.2,1.35) {$X_1\sim\mu_1$};
  \node (y0) at (0,-1.35) {$\nu_0\sim Y_0$};
  \node (y1) at (7.2,-1.35) {$Y_1\sim\nu_1$};

  \draw[forward] (x0) -- node[maplabel,above] {$\nabla\Phi$} (x1);
  \draw[forward] (x0) to[bend left=5]
    node[maplabel,pos=.54,sloped,above] {$\nabla b$} (y1);
  \draw[backward] (y1) to[bend left=5]
    node[maplabel,pos=.46,sloped,below] {$\nabla b^*$} (x0);
  \draw[backward] (y1) -- node[maplabel,below] {$\nabla\Phi^*$} (y0);
\end{tikzpicture}
\end{center}

Apply Lemma~\ref{lem:Kolesnikov-two-map} first with source $\nu_1$ and maps $\nabla b^*$ and $\nabla \Phi^*$, then with source $\mu_0$ and maps $\nabla \Phi$ and $\nabla b$. All four maps are gradients of convex functions. By \eqref{eq:three-entropies},
\[
 \mathbb E|X_0-Y_0|^2\leq\frac{2H}{\kappa_0},
 \qquad
 \mathbb E|X_1-Y_1|^2\leq\frac{2H}{\kappa_1}.
\]
Since $(1-t)X_0+tX_1$ and $(1-t)Y_0+tY_1=F_t(Y_0)$ have laws $\mu_t$ and $\nu$, respectively, Minkowski's inequality yields
\begin{align*}
 W_2(\mu_t,\nu)
 &\leq\bigl\|(1-t)(X_0-Y_0)+t(X_1-Y_1)\bigr\|_{L^2}\\
 &\leq\sqrt{2H}\left(\frac{1-t}{\sqrt{\kappa_0}}+
                         \frac{t}{\sqrt{\kappa_1}}\right).
\end{align*}
Finally, for $\mu_i=N(0,\kappa_i^{-1}I)$, the interpolant is $N(0,c_t^2I)$, where $c_t=(1-t)\kappa_0^{-1/2}+t\kappa_1^{-1/2}$, and its optimal Talagrand constant is $2c_t^2$ (equality in  Talagrand inequality is achieved for $\nu = N(m,c_t^2 I)$ for any $m\in \R^n$). This proves sharpness.
\end{proof}

\section{Convexity of transport entropy constants on the line}\label{sec:transportentropy}

We denote by $\lambda$ the Lebesgue measure on $(0,1)$. For a probability measure $\mu$ on $\R$, let $Q_\mu:(0,1)\to\R$ be its (left-continuous) quantile function, defined by
\[
 Q_\mu(u)=\inf\{x\in\R : \mu((-\infty,x])\geq u\},\qquad u\in(0,1).
\]
It is well known that $(Q_\mu)_\#\lambda=\mu$. Given two probability measures $\mu_0,\mu_1$ on $\R$, with quantile functions $Q_0,Q_1$, we call \emph{monotone interpolation} between $\mu_0$ and $\mu_1$ the family $(\mu_t)_{t\in[0,1]}$ defined by
\begin{equation}\label{eq:quantile-geodesic}
 Q_t=(1-t)Q_0+tQ_1,
 \qquad
 \mu_t=(Q_t)_\#\lambda.
\end{equation}
Recall that a monotone coupling on $\R$ is optimal for every cost of the form $\alpha(x-y)$ with $\alpha$ convex, see for instance \cite[Chapter~2]{Villani2009}. In particular, $(Q_0,Q_1)_\#\lambda$ is an optimal coupling of $\mu_0$ and $\mu_1$ for such costs and, if $\mu_0,\mu_1\in\Pp_p(\R)$ for some $p\geq1$, the monotone interpolation is a constant speed geodesic for $W_p$.

Let $\alpha:\R\to[0,\infty)$ be an even convex function with $\alpha(0)=0$, and let $\psi:[0,\infty)\to[0,\infty)$ be non-decreasing, with $\psi(0)=0$ and $\psi(s)>0$ for $s>0$. We consider the optimal constant
\begin{equation}\label{eq:Kalpha-def}
 \mathsf T_{\alpha,\psi}(\mu)
 =\sup_{0<\Ent(\nu\mid\mu)<\infty}
 \frac{\mathcal T_\alpha(\mu,\nu)}{\psi(\Ent(\nu\mid\mu))},
 \qquad \mu\in\Pp(\R).
\end{equation}
Thus $\mathsf T_{\alpha,\psi}(\mu)$ is the least constant in the inequality
$\mathcal T_\alpha(\mu,\nu)\leq C\psi(\Ent(\nu\mid\mu))$ for competitors of finite relative entropy. For instance, choosing $\alpha(z)=|z|$ and $\psi(s)=\sqrt{s}$ corresponds to the usual $T_1$ inequality.

\begin{lemma}\label{lem:entropy-lift}
Let $Q:(0,1)\to\R$ be non-decreasing and let $\mu=Q_\#\lambda$. Then
\begin{equation}\label{eq:lift-formula}
 \mathsf T_{\alpha,\psi}(\mu)
 =\sup_{0<\Ent(\rho\mid\lambda)<\infty}
 \frac{\mathcal T_\alpha(Q_\#\lambda,Q_\#\rho)}
      {\psi(\Ent(\rho\mid\lambda))},
\end{equation}
where the supremum runs over probability measures $\rho$ on $(0,1)$.
\end{lemma}

\begin{proof}
For every probability measure $\rho$ on $(0,1)$, the data processing inequality gives
\[
 \Ent(Q_\#\rho\mid\mu)\leq\Ent(\rho\mid\lambda).
\]
Since $\psi$ is non-decreasing, every quotient on the right-hand side of \eqref{eq:lift-formula} is at most $\mathsf T_{\alpha,\psi}(\mu)$. 
Conversely, let $\nu=f\mu$ satisfy $0<\Ent(\nu\mid\mu)<\infty$, and define $\rho$ by
\[
 \frac{\mathrm d\rho}{\mathrm d\lambda}(u)=f(Q(u)).
\]
Then $Q_\#\rho=\nu$ and
\[
 \Ent(\rho\mid\lambda)
 =\int_0^1 f(Q(u))\log f(Q(u))\dd u
 =\int f\log f\dd\mu
 =\Ent(\nu\mid\mu).
\]
Therefore, for any probability measure $\nu$ on $\R$ with $0<\Ent(\nu \mid \mu)<+\infty$, there exists $\rho \in \Pp((0,1))$ such that $Q_\# \rho = \nu$ and $\Ent(\nu \mid \mu) = \Ent(\rho \mid \lambda)$, which concludes the proof.
\end{proof}

\begin{theorem}[Convexity for Talagrand type inequalities]\label{thm:Kalpha-convex}
Let $\mu_0,\mu_1$ be arbitrary probability measures on $\R$, and let $(\mu_t)$ be their monotone interpolation. Then $t\mapsto\mathsf T_{\alpha,\psi}(\mu_t)$ is convex. In particular,
\begin{equation}\label{eq:Kalpha-convex}
 \mathsf T_{\alpha,\psi}(\mu_t)
 \leq(1-t)\mathsf T_{\alpha,\psi}(\mu_0)
       +t\mathsf T_{\alpha,\psi}(\mu_1),
 \qquad t\in(0,1).
\end{equation}
\end{theorem}

\begin{proof}
Write $Q_t=(1-t)Q_0+tQ_1$ for the quantile function of $\mu_t$. Fix a probability measure $\rho$ on $(0,1)$ with $0<\Ent(\rho\mid\lambda)<\infty$, and let $R$ be its quantile function. Since $Q_t$ is non-decreasing, $(Q_t,Q_t\circ R)_\#\lambda$ is a monotone coupling of $\mu_t$ and $(Q_t)_\#\rho$. It is optimal for the convex cost $\alpha$, so
\begin{equation}\label{eq:cost-quantile}
 \mathcal T_\alpha(\mu_t,(Q_t)_\#\rho)
 =\int_0^1\alpha\bigl(Q_t(u)-Q_t(R(u))\bigr)\dd u.
\end{equation}
The argument of $\alpha$ is affine in $t$. Hence the right-hand side of \eqref{eq:cost-quantile} is convex in $t$. Dividing by $\psi(\Ent(\rho\mid\lambda))$, which is positive and independent of $t$, and taking the supremum over $\rho$ proves the result by Lemma~\ref{lem:entropy-lift}, since a supremum of convex functions is convex.
\end{proof}

Recall that, for $\mu\in\Pp(\R)$ and $1\leq p\leq2$, we denote by $\mathsf T_p(\mu)$ the best constant in the inequality $W_p^2(\mu,\nu)\leq\mathsf T_p(\mu)\Ent(\nu\mid\mu)$, for all $\nu \in \mathcal{P}(\R)$.
\begin{corollary}\label{cor:Talagrand}
Along the monotone interpolation $(\mu_t)$ between any two probability measures $\mu_0,\mu_1$ on $\R$, the function $t\mapsto\sqrt{\mathsf T_p(\mu_t)}$ is convex. In particular,
\begin{equation}\label{eq:Tp-scale}
 \sqrt{\mathsf T_p(\mu_t)}
 \leq(1-t)\sqrt{\mathsf T_p(\mu_0)}
       +t\sqrt{\mathsf T_p(\mu_1)},\qquad t\in(0,1).
\end{equation}
\end{corollary}

\begin{proof}
For $\alpha(z)=|z|^p$ and $\psi(s)=s^{p/2}$, Lemma~\ref{lem:entropy-lift} and \eqref{eq:cost-quantile} give
\[
 \sqrt{\mathsf T_p(\mu_t)}
 =\sup_{0<\Ent(\rho\mid\lambda)<\infty}
 \frac{\|Q_t-Q_t\circ R\|_{L^p(0,1)}}
      {\sqrt{\Ent(\rho\mid\lambda)}},
\]
where $R$ denotes the quantile function of $\rho$. Each numerator is the norm of a function affine in $t$, and is therefore convex by Minkowski's inequality. The denominator is independent of $t$, so the supremum is convex as well.
\end{proof}

\begin{remark}\label{rem:homogeneity}
One can also deduce Corollary~\ref{cor:Talagrand} by homogeneity. Fix $t\in(0,1)$. If both endpoint constants are finite, choose $a_i>\sqrt{\mathsf T_p(\mu_i)}$ and set $c=(1-t)a_0+ta_1$ and $s=ta_1/c$. Denote $D_a(x)=ax$ the dilation of ratio $a>0$. The measures $(D_{1/a_i})_\#\mu_i$ have Talagrand constant $\mathsf T_p$ smaller than $1$ by homogeneity, and their monotone interpolant at time $s$ is $(D_{1/c})_\#\mu_t$. Theorem~\ref{thm:Kalpha-convex}, applied to this monotone interpolant with $\alpha(z)=|z|^p$ and $\psi(s)=s^{p/2}$ gives, after rescaling, $\mathsf T_p(\mu_t)\leq c^2 = ((1-t)a_0+ta_1)^2$. Letting $a_i\downarrow\sqrt{\mathsf T_p(\mu_i)}$ proves \eqref{eq:Tp-scale}, including when an endpoint constant is zero. This normalization argument is classical, see \cite[proof of Theorem~2.3]{McCann1997}.
\end{remark}

\section{Convexity of \texorpdfstring{$L^p$}{Lp} Poincar\'e constants on the line}\label{sec:sobolev}

Recall that $\mathsf C_p(\mu)$ is defined as the best constant in \eqref{eq:intro-Cp}. The main result of this section is the following.

\begin{theorem}[Convexity for Poincar\'e type inequalities]\label{thm:sobolev-arbitrary}
Let $\mu_0,\mu_1$ be arbitrary probability measures on $\R$ and let $(\mu_t)_{t\in[0,1]}$ be their monotone interpolation. For every $1\leq p<\infty$, the function $t\mapsto\mathsf C_p(\mu_t)^{1/p}$ is convex on $[0,1]$. In particular,
\begin{equation}\label{eq:Cp-arbitrary}
 \mathsf C_p(\mu_t)^{1/p}
 \leq(1-t)\mathsf C_p(\mu_0)^{1/p}
          +t\mathsf C_p(\mu_1)^{1/p},\qquad t\in(0,1).
\end{equation}
For $p=2$ and $p=1$, this reads respectively
\begin{equation}\label{eq:poincare-scale}
 \sqrt{C_P(\mu_t)}
 \leq(1-t)\sqrt{C_P(\mu_0)}+t\sqrt{C_P(\mu_1)},\qquad t\in(0,1)
\end{equation}
and
\begin{equation}\label{eq:cheeger-scale}
 \frac1{h(\mu_t)}
 \leq\frac{1-t}{h(\mu_0)}+\frac{t}{h(\mu_1)},\qquad t\in(0,1).
\end{equation}
\end{theorem}
Note that the displacement convexity of $\mathsf C_p^{1/p}$ implies the same for $ \mathsf C_p$. The homogeneity argument of Remark~\ref{rem:homogeneity} shows that this is in fact equivalent.
We first prove Theorem~\ref{thm:sobolev-arbitrary} for measures with positive continuous densities on $\R$, and then pass to arbitrary measures by convolution.

\subsection{Proof of Theorem \ref{thm:sobolev-arbitrary} in the regular case}\label{sub:regular}
\label{sec:proof_regular}
In this section, we suppose that $\mu_0,\mu_1$ have positive continuous densities on $\R$ and we show the convexity of $t \mapsto \mathsf C_p^{1/p}(\mu_t)$ in this case.

The monotone transport $T:\R\to\R$ is an increasing $C^1$ diffeomorphism. Set $F_t=(1-t)\id+tT$, so that $\mu_t=(F_t)_\#\mu_0$.
Each $F_t$ is an increasing $C^1$ diffeomorphism, and $\mu_t$ has a positive continuous density on $\R$. We have by definition
\begin{equation}\label{eq:Cp-def}
 \mathsf C_p(\mu_t)^{1/p}
 =\sup_{\substack{g\in C^1(\R),\,g'\in C_c(\R)\\ \int|g'|^p\dd\mu_t>0}}
 \frac{\inf_{a\in\R}\|g-a\|_{L^p(\mu_t)}}{\|g'\|_{L^p(\mu_t)}}.
\end{equation}
Fix $x_0\in\R$. For $f\in C^1(\R)$ with $f'\in C_c(\R)$, define $f_t$ by
\begin{equation}\label{eq:transport-derivative}
 f_t(F_t(x_0))=0,\qquad f_t'=f'\circ F_t^{-1}.
\end{equation}
Note that this construction is the same as in the proof of Aishwarya and Rotem \cite[Theorem~4.11]{AishwaryaRotem2026}. Then $f_t\in C^1(\R)$ with $f_t'\in C_c(\R)$ and
\begin{equation}\label{eq:energy-frozen}
 \|f_t'\|_{L^p(\mu_t)}=\|f'\|_{L^p(\mu_0)}.
\end{equation}
Each test function $g$ in \eqref{eq:Cp-def} is obtained for some function $f$, by taking $f'=g'\circ F_t$, which yields $f_t=g+c$ for some $c\in\R$. Thus we have
\begin{equation}\label{eq:fixed-sup}
 \mathsf C_p(\mu_t)^{1/p}
 =\sup_{\substack{f\in C^1(\R),\,f'\in C_c(\R)\\ \|f'\|_{L^p(\mu_0)}>0}}
 \frac{\inf_{a\in\R}\|f_t\circ F_t-a\|_{L^p(\mu_0)}}
      {\|f'\|_{L^p(\mu_0)}}.
\end{equation}
For each such $f$, a change of variables yields
\[
 f_t(F_t(x))=\int_{x_0}^x f'(r)F_t'(r)\dd r=(1-t)(f(x)-f(x_0))+t\int_{x_0}^x f'(r)T'(r)\dd r.
\]
The function $f_t \circ F_t - a$ depends affinely on $(t,a)$. The composition with $\|\cdot\|_ {L^p(\mu_0)}$ gives a  quantity that is jointly convex in $t$ and $a$. Therefore, as a partial infimum of a jointly convex function, the numerator in \eqref{eq:fixed-sup} is convex in $t$. The denominator and the set of competitors do not depend on $t$, so their supremum is convex.

\subsection{Proof of Theorem \ref{thm:sobolev-arbitrary} in the general case}\label{sub:general}

Fix $1\leq p<\infty$ and write
\[
S_p(\mu)=\mathsf C_p(\mu)^{1/p},\qquad
 E_p(f;\mu)=\inf_{a\in\R}\int|f-a|^p\dd\mu.
\]
We first show that $S_p$ is lower semicontinuous for weak convergence. Fix $f\in C^1(\R)$ with $f'\in C_c(\R)$. Since $|f'|^p$ is bounded and continuous, the map $\mu\mapsto\int|f'|^p\dd\mu$ is continuous for the weak topology. The same is true for the map $\mu\mapsto E_p(f;\mu)$: indeed, the infimum defining $E_p(f;\mu)$ can be restricted to $a\in[\min f,\max f]$, and, along a weakly converging sequence, the convergence of the integrals $\int|f-a|^p\dd\mu$ is uniform with respect to $a$ in this compact interval. Now, if $\mu_k\to\mu$ weakly, it is enough to pass to the limit in the inequality
\[
 E_p(f;\mu_k)\leq S_p(\mu_k)^p\int|f'|^p\dd\mu_k
\]
along a subsequence realizing $\liminf_k S_p(\mu_k)$.

For $\varepsilon>0$, let $\eta_\varepsilon(\mathrm dy)=(2\varepsilon)^{-1}e^{-|y|/\varepsilon}\dd y$. We temporarily assume the following estimate, which is proved below:
\begin{equation}\label{eq:sobolev-convolution}
 S_p(\mu*\eta_\varepsilon)\leq S_p(\mu)+2p\varepsilon.
\end{equation}
Fix $t\in(0,1)$. We may assume that $S_p(\mu_0)$ and $S_p(\mu_1)$ are finite. Set $\mu_i^\varepsilon=\mu_i*\eta_\varepsilon$ and let $(\mu_t^\varepsilon)$ be their monotone interpolation. Since the endpoint densities are positive and continuous on $\R$, Section~\ref{sec:proof_regular} gives
\[
 S_p(\mu_t^\varepsilon)
 \leq(1-t)S_p(\mu_0^\varepsilon)+tS_p(\mu_1^\varepsilon)
 \leq(1-t)S_p(\mu_0)+tS_p(\mu_1)+2p\varepsilon.
\]
As $\varepsilon\to0$, the endpoint quantiles converge almost everywhere to those of $\mu_0$ and $\mu_1$. Hence $\mu_t^\varepsilon$ converges weakly to $\mu_t$. Lower semicontinuity allows us to pass to the limit in the inequality and gives \eqref{eq:Cp-arbitrary}.

\begin{proof}[Proof of \eqref{eq:sobolev-convolution}]
Let us begin with a technical result on the measure $\eta_\varepsilon$: for every bounded function $h \in C^1(\mathbb R)$, we have
\begin{equation}\label{eq:eta_eps}
    \left\|h-\int h\dd\eta_\varepsilon\right\|_{L^p(\eta_\varepsilon)}
 \leq2p\varepsilon\|h'\|_{L^p(\eta_\varepsilon)}.
\end{equation}
Indeed, for such a function $h$, by the triangle inequality and H\"older's inequality,
\[ \left\|h-\int h\dd\eta_\varepsilon\right\|_{L^p(\eta_\varepsilon)}
 \leq \left\|h-h(0)\right\|_{L^p(\eta_\varepsilon)} + |h(0)-\int h\dd\eta_\varepsilon | \leq 2\left\|h-h(0)\right\|_{L^p(\eta_\varepsilon)}. \]
Moreover, an integration by parts on each half-line, followed by H\"older's inequality, gives
\[
 \int |h-h(0)|^p\dd\eta_\varepsilon
 \leq p\varepsilon\int |h-h(0)|^{p-1}|h'|\dd\eta_\varepsilon
 \leq p\varepsilon\|h-h(0)\|_{L^p(\eta_\varepsilon)}^{p-1}\|h'\|_{L^p(\eta_\varepsilon)}.
\]
Hence $\|h-h(0)\|_{L^p(\eta_\varepsilon)}\leq p\varepsilon\|h'\|_{L^p(\eta_\varepsilon)}$, which proves \eqref{eq:eta_eps}.

Now let $X\sim\mu$ and $Y\sim\eta_\varepsilon$ be independent random variables, so that $X+Y\sim\mu*\eta_\varepsilon$, fix $f\in C^1(\R)$ with $f'\in C_c(\R)$ and set $g(x)=\mathbb E f(x+Y)$, $x\in\R$. Applying \eqref{eq:eta_eps} to $h=f(x+\cdot\,)$ for each fixed $x$, integrating in $x$ and using Minkowski's inequality, we obtain
\begin{align*}
 \inf_{a\in\R}\|f(X+Y)-a\|_p
 &\leq\|f(X+Y)-g(X)\|_p+\inf_{a\in\R}\|g(X)-a\|_p\\
 &\leq2p\varepsilon\|f'(X+Y)\|_p+S_p(\mu)\|g'(X)\|_p\\
 &\leq\bigl(2p\varepsilon+S_p(\mu)\bigr)\|f'(X+Y)\|_p.
\end{align*}
In the second line, we used the inequality defining $S_p(\mu)$ for the function $g$, which is legitimate after a spatial cutoff since $g$ and $g'$ are bounded, as explained in footnote~\ref{fn:test-class}. In the last line, we used Jensen's inequality, since $g'(X)=\mathbb E[f'(X+Y)\mid X]$. This proves \eqref{eq:sobolev-convolution}.
\end{proof}

\section{A counterexample in dimension two}\label{sec:counter-example}

We construct two probability measures on $\R^2$ which obey $L^p$ Poincaré and Talagrand $T_2$ inequalities, while every interior interpolant fails to satisfy those inequalities. The obstruction is geometric: the interpolant's support has two connected components. 

For $t\in[0,1]$, define $\Gamma_t^+=[(-1,t),(1,1-t)]$, $\Gamma_t^-=[(-1,-t),(1,t-1)]$ and $\Gamma_t=\Gamma_t^+\cup\Gamma_t^-$. We denote by $\mu_0$ and $\mu_1$ the normalized arclength measures on the chevrons $\Gamma_0$ and $\Gamma_1$ respectively.

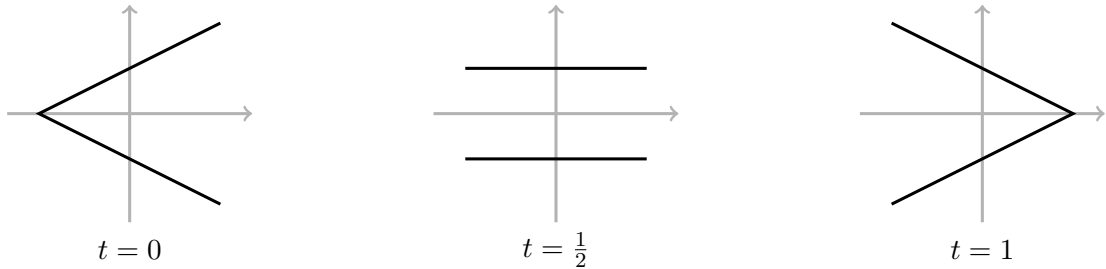
\begin{figure}[ht]
\centering
\begin{tikzpicture}[scale=1.2, line width=1.15pt]
  \begin{scope}[xshift=-4.7cm]
    \draw[->,gray!60] (-1.35,0)--(1.35,0);
    \draw[->,gray!60] (0,-1.2)--(0,1.2);
    \draw (1,1)--(-1,0)--(1,-1);
    \node at (0,-1.5) {$t=0$};
  \end{scope}
  \begin{scope}
    \draw[->,gray!60] (-1.35,0)--(1.35,0);
    \draw[->,gray!60] (0,-1.2)--(0,1.2);
    \draw (-1,0.5)--(1,0.5);
    \draw (-1,-0.5)--(1,-0.5);
    \node at (0,-1.5) {$t=\tfrac12$};
  \end{scope}
  \begin{scope}[xshift=4.7cm]
    \draw[->,gray!60] (-1.35,0)--(1.35,0);
    \draw[->,gray!60] (0,-1.2)--(0,1.2);
    \draw (-1,1)--(1,0)--(-1,-1);
    \node at (0,-1.5) {$t=1$};
  \end{scope}
\end{tikzpicture}
\caption{The quadratic Wasserstein geodesic between the two chevrons.  At time $t=1/2$, the support $\Gamma_{1/2}$ is the union of two horizontal segments.}
\label{fig:chevrons}
\end{figure}

\begin{proposition}[Disconnection of the support]\label{prop:chevron}
There exists a unique quadratic optimal coupling between $\mu_0$ and $\mu_1$. For every $t\in[0,1]$, the displacement interpolation $\mu_t$ is the normalized arclength measure on $\Gamma_t$.
\end{proposition}

\begin{proof}
Let $\pi$ be a quadratic optimal coupling. We shall show that $\pi$ preserves the upper and lower branches. Let $R$ be the vertical reflection defined by $R(x_1,x_2)=(x_1,-x_2)$.  The coupling $\bar\pi=\tfrac12(\pi+(R,R)_\#\pi)$ has the same cost as $\pi$. Since both marginals are $R$-invariant, $\bar \pi$ also has $\mu_0$ and $\mu_1$ as marginals. It is hence optimal, and $R$-invariant by construction. For every \((x,y)\in\operatorname{spt}\bar\pi\), invariance gives \((Rx,Ry)\in\operatorname{spt}\bar\pi\). By monotonicity of optimal couplings, we have
\[
(x-Rx)\cdot(y-Ry)\ge0.
\]
Since \(x-Rx=(0,2x_2)\) and \(y-Ry=(0,2y_2)\), we obtain \(x_2y_2\ge0\), so that \(\bar\pi\) is supported on pairs lying in the same closed half-plane. Since $\pi$ is absolutely continuous with respect to $\bar\pi$, the same holds for \(\pi\). Thus $\pi$ preserves the upper and lower branches: $\pi(\Gamma_0^+\times\Gamma_1^+)=\pi(\Gamma_0^-\times\Gamma_1^-)=1/2$. The restriction $2\pi_{|\Gamma_0^+\times\Gamma_1^+}$ is an optimal coupling of the uniform measures on $\Gamma_0^+$ and $\Gamma_1^+$.

Consider, in general, two probability measures, $\nu_0$ and $\nu_1$,  supported on $\Gamma_0^+$ and $\Gamma_1^+$ respectively. If $X$ has law $\nu_0$, by the geometry of $\Gamma_0^+$, we have $X = (U,\tfrac12 (1+U))$ with $U$ the first coordinate of $X$.   Similarly, $Y \sim \nu_1$ can be expressed as  $(V,\tfrac12(1-V))$ and we have
\[ \mathbb{E}|X-Y|^2 =  \frac{3}{4}\mathbb{E}|U - V|^2 + \frac{1}{2} \mathbb{E}U^2 + \frac{1}{2} \mathbb{E}V^2. \]
The moments $\mathbb{E}U^2$ and $\mathbb{E}V^2$ being independent of the coupling chosen, minimizing $\mathbb{E}|X-Y|^2$ among the couplings $(X,Y)$ of $\nu_0$ and $\nu_1$ is equivalent to minimize $\frac{3}{4}\mathbb{E}|U - V|^2$ among the couplings of $\tilde \nu_0$ and $\tilde \nu_1$, defined as the pushforward of $\nu_0$ and $\nu_1$ by the first coordinate. This one dimensional optimal transport problem is uniquely solved by the monotone coupling.

In our case, $\nu_0$ and $\nu_1$ are uniform on $\Gamma_0^+$ and $\Gamma_0^-$, so $\tilde \nu_0$ and $\tilde \nu_1$ are uniform on $[-1,1]$. Since they are equal, their optimal coupling is given by $U = V$ and so the optimal coupling $2\pi_{|\Gamma_0^+\times\Gamma_1^+}$ is uniform on the set $\{(x,y)\in \Gamma_0^+\times\Gamma_1^+ : x_1=y_1\}$. This shows, by symmetry, that $\pi$ is uniform on the set
\[
 \{(x,y)\in\Gamma_0\times\Gamma_1 : x_1=y_1 \quad \text{and}\quad x_2y_2\ge0\}.
\]
As $\mu_t = ((1-t)x+ty)_\#\pi$, we obtain that $\mu_t$ is the normalized arclength measure on $\Gamma_t$.
\end{proof}

\begin{corollary}[Loss of functional inequalities in dimension at least two]\label{cor:chevron-functional}
Let $p \in [1,\infty)$. With $\mathsf C_p$ defined by \eqref{eq:intro-Cp}, the chevron measures $\mu_0$ and $\mu_1$ satisfy
\[
 \mathsf C_p(\mu_i), \mathsf T_2(\mu_i) <\infty.
\]
For every $0<t<1$, however,
\[
 \mathsf C_p(\mu_t) = \mathsf T_2(\mu_t)=\infty.
\]
Thus, $L^p$-Poincaré inequality and Talagrand $T_2$ inequality fail to propagate by displacement interpolation in dimension $n\geq2$.
\end{corollary}

\begin{proof}
Each chevron is the image of an interval of length $L=2\sqrt5$ under a unit-speed parametrization $\gamma$, which sends $\lambda_L$, the uniform measure on $[0,L]$,  to its normalized arclength measure. For $f\in C^1(\R^2)$ with $\nabla f\in C_c(\R^2;\R^2)$, the function $g=f\circ\gamma$ is absolutely continuous and satisfies $|g'|\leq|\nabla f|\circ\gamma$ almost everywhere. Since
\[
\inf_{a\in\mathbb R}\|f-a\|_{L^p(\mu_i)}
=
 \inf_{a\in\R}\|g-a\|_{L^p(\lambda_L)}
 \leq L\|g'\|_{L^p(\lambda_L)}
 \leq L\|\nabla f\|_{L^p(\mu_i)},
\]
we obtain $\mathsf C_p(\mu_i)\leq L^p$.

Fix $0<t<1$. By Proposition~\ref{prop:chevron}, the support of $\mu_t$ consists of two compact sets $\Gamma_t^+$ and $\Gamma_t^-$ at positive distance, each of mass $1/2$. Choose $f\in C^1(\R^2)$ with $\nabla f\in C_c(\R^2;\R^2)$ equal to $1$ near $\Gamma_t^+$ and to $-1$ near $\Gamma_t^-$. Then
\[
 \inf_{a\in\R}\int|f-a|^p\dd\mu_t=1,
 \qquad \int|\nabla f|^p\dd\mu_t=0,
\]
so $\mathsf C_p(\mu_t)=\infty$ for every finite $p\geq1$.

It remains to consider $T_2$. For compactly supported probability measures, finiteness of the Poincar\'e and $T_2$ constants is equivalent. Indeed, by \cite[Corollary~5.1]{BobkovGentilLedoux2001}, a Poincar\'e inequality implies a transport-entropy inequality for a cost that is quadratic near zero and linear at infinity. On a bounded support this cost bounds a positive multiple of $|x-y|^2$, which gives $T_2$. The converse follows by linearization. Applying this equivalence to $\mu_0$, $\mu_1$ and $\mu_t$ proves the remaining assertions. Finally, the construction embeds isometrically in $\R^n$ for every $n\geq2$.
\end{proof}

\section*{Acknowledgements}
Nathael Gozlan and Hugo Malamut acknowledge support from the Agence nationale de la recherche (ANR) through grant ANR-23-CE40-0017 (project SOCOT). This research was conducted within the FP2M federation (CNRS FR~2036).

\section*{Declaration of generative AI use}
The proof of Theorem \ref{thm:curvature-T2} was suggested by ChatGPT (OpenAI) in response to questions posed by the authors. The other mathematical results and their proofs were obtained in earlier work without the use of generative AI. ChatGPT was also used to assist with the mathematical and editorial review of the manuscript and with bibliographic checks. The authors independently verified the proof of Theorem \ref{thm:curvature-T2} and critically assessed the suggestions made during the revision process. They have reviewed and approved the final manuscript and take full responsibility for its mathematical correctness and overall content.

\end{document}